\documentclass{amsart}
\usepackage{amssymb}
\usepackage[matrix,arrow]{xy}

\newtheorem{theorem}{Theorem}[section]
\newtheorem{lemma}[theorem]{Lemma}
\newtheorem{corollary}[theorem]{Corollary}

\newtheorem{proposition}[theorem]{Proposition}

\newenvironment{proofprop}{\vspace{1ex}\noindent{\sl Proof of the proposition.}
\hspace{0.5em}}{\hfill\qed\vspace{1ex}}

\def\GL{\mathrm{GL}}
\def\Mat{\mathrm{Mat}}
\def\C{\mathbf{C}}
\def\Z{\mathbf{Z}}
\def\F{\mathbf{F}}
\def\Q{\mathbf{Q}}
\def\C{\mathbf{C}}
\def\gl{\mathfrak{gl}}

\begin{document}
\centerline{}

\title{The center of pure complex braid groups}
\author[F.~Digne]{Fran\c cois Digne}
\address{LAMFA, Universit\'e de Picardie-Jules Verne}
\email{digne@u-picardie.fr}
\author[I.~Marin]{Ivan Marin}
\address{IMJ, Universit\'e Paris VII}
\email{marin@math.jussieu.fr}
\author[J.~Michel]{Jean Michel}
\address{IMJ, Universit\'e Paris VII}
\email{jmichel@math.jussieu.fr}
\medskip

\begin{abstract}
Brou\'e, Malle and Rouquier conjectured in \cite{BMR} that
the center of the pure braid group of an irreducible finite complex reflection
group is cyclic. We prove this conjecture, for the remaining exceptional
types, using the analogous result for
the full braid group due to Bessis, and we actually prove the
stronger statement that any finite index subgroup of such braid group
has cyclic center.
\end{abstract}

\maketitle

\section{Introduction}

Let $W$ denote a (finite) complex reflection group of rank $n \geq 1$, that is
a  finite  subgroup  of  $\GL_n(\C)$  generated by pseudo-reflections, that is
elements  of  $\GL_n(\C)$  which  fix  an  hyperplane.  We  assume  $W$  to be
irreducible, in order to simplify statements.
To such a group
is associated an hyperplane arrangement $\mathcal{A}$, made of the
collection of the fixed hyperplanes associated to the pseudo-reflections
in $W$. Let $X = \C^n \setminus \bigcup \mathcal{A}$ denote the hyperplane
complement. The groups $P = \pi_1(X)$ and $B = \pi_1(X/W)$
are known as the pure braid group and braid group associated to $W$.
There is a short exact sequence $1 \to P \to B \to W \to 1$. In the
case $W$ is a real reflection group, the group $B$
is an Artin-Tits group, with prototype the usual braid group $B$ associated
to $W = \mathfrak{S}_n$. An extension of many results on Artin-Tits groups
to the more general setup introduced here has been proposed in \cite{BMR}.
Several of them are still conjectural, such as
the
determination of the center of these groups. The goal of this
note is to clarify the status of this question.
Recall that in \cite{BMR} are
introduced (infinite order) elements $\beta \in Z(B)$, $\pi \in Z(P)$,
with $\beta^{|Z(W)|} = \pi$ such that the image of $\beta$ in $W$ 
generates $Z(W)$.

Our purpose is to summarize what can be stated on this topic,
as follows:

\begin{theorem} \label{thzb} The center $Z(B)$ is infinite cyclic, and
generated by $\mathbf{\beta}$.
\end{theorem}

\begin{theorem} \label{thzp}The center $Z(P)$ is infinite cyclic, and
generated by $\mathbf{\pi}$.
\end{theorem}

\begin{theorem} \label{thse}There exists a short exact sequence
$$
1 \to Z(P) \to Z(B) \to Z(W) \to 1
$$
\end{theorem}

These  three results were conjectured in  \cite{BMR}, and proved there for the
infinite  series $G(de,e,n)$, as well as for the (easy) case of groups of rank
2.  Actually, for  a given  group $W$,  theorem \ref{thzp}  implies by general
arguments  theorem  \ref{thse}  and  theorem  \ref{thzb},  as every element in
$Z(W)$  can be lifted  to an element  in $Z(B)$. The  remaining cases were the
exceptional  groups $G_{24}$,  $G_{25}$, $G_{26}$,  $G_{27}$, $G_{29}$, $G_{31}$, $G_{32}$,
$G_{33}$,  $G_{34}$ in  Shephard-Todd notation.  As noted  in \cite{BMR}, theorem
\ref{thzb}  for the  so-called ``Shephard''  groups $G_{25},  G_{26}, G_{32}$ is
true  because they have the same $B$  as some Coxeter group (however, contrary
to  what is claimed in  \cite{BMR}, this does not  prove the other two results
!).  Note also that in \cite[proposition 2.23]{BMR} the bottom-right square in
the diagram is not commutative and the bottom sequence is not exact.

Theorem \ref{thzb} was proved by D. Bessis in \cite{BESSIS}
for the well-generated groups $G_{24}$, $G_{25}$, $G_{26}$, $G_{27}$,
$G_{29}$,  $G_{32}$, $G_{33}$, $G_{34}$. The remaining case of $G_{31}$ can be
obtained by a short argument due to D. Bessis (personal communication)
that we reproduce here for the convenience of the reader (see section 3).

Once theorem \ref{thzb} is known, theorem \ref{thse} reduces
to theorem \ref{thzp}, and more precisely to the statement $Z(P) \subset Z(B)$.
Actually we prove the following stronger theorem in the subsequent sections.

\medskip

\begin{theorem} \label{thff} If $U$ is a finite index subgroup of $B$,
then $Z(U) \subset Z(B)$.
\end{theorem}



\section{Garside theory}

A Garside monoid is a cancellative monoid which is
generated by its atoms (that is 
elements which have no proper divisors), such that any two elements
have a least common right-multiple and a greatest common left-divisor
and such that there exist an element $\Delta$ (a ``Garside element'') whose
left- and right-divisors are the same and generate the monoid.
Here $a$ left-divides $b$, which will be denoted by $a\preccurlyeq b$,
means that there exists $c$ with $b=ac$ and similarly for right-divisibility.
We refer to \cite{DEHORNOY} or \cite{DIGNEMICHEL} for the basic notions on Garside theory.
We will use in particular the following property: if $M$ is a Garside monoid,
let $\alpha(x)$ denote the left-gcd of an element $x\in M$ with $\Delta$, then
$\alpha(xy)=\alpha(x\alpha(y))$ for any $y\in M$.

In this section, we consider a Garside monoid $M$ which satisfies
in addition the following properties:

\begin{enumerate}
\item[(i)] There is an additive length function on $M$ such that the atoms
have length 1.

This implies that the atoms are precisely the elements of length 1 and that
the only element of length 0 is 1.

\item[(ii)]For every couple of atoms $s \neq t$, their right-lcm $\Delta_{s,t}$
is balanced (meaning that its left- and right-divisors are the same).
\item[(iii)] For any couple of atoms $s \neq t$ and any positive integer $n$,
the left-gcd of $\Delta_{s,t}$ and $s^n$ is equal to $s$.
\end{enumerate}

Note that many of the monoids which have been studied
for complex braid groups and for Artin-Tits groups satisfy the above conditions.
In particular we have
\begin{proposition}\label{monoid list} Let $M$ be one of the following monoids:
\begin{enumerate}
\item[(M1)]
The classical monoid of positive elements in the Artin-Tits groups
associated to finite Coxeter groups (see \cite{DEHORNOY}).
\item[(M2)]
The ``dual'' monoid of \cite{BESSIS}.
\item[(M3)]
The ``parachute'' monoid of \cite{CORRANPICANTIN}. 
\item[(M4)] The ``dual'' monoids for Artin-Tits groups of type $\tilde A$ and
$\tilde C$ of \cite{DIGNE2} and \cite{DIGNE3}.
\item[(M5)] The monoids $f(h,m)$ (for $h,m \geq 1$) presented by generators $x_1,\dots,x_m$ and relations
$$
\underbrace{x_1 x_2 \dots x_m x_1 \dots}_{h \mbox{ terms}} = 
\underbrace{x_2 x_3\dots x_m x_1 \dots}_{h \mbox{ terms}}=  \dots
$$
\end{enumerate}
Then $M$ is a Garside monoid and satisfies (i), (ii) and (iii).
\end{proposition}
\begin{proof}
In case (M1) since in the classical monoid we have
$\Delta_{s,t}=\underbrace{sts\ldots}_{e}=\underbrace{tst\ldots}_{e}$  for some
$e\ge  2$ and these are the only  decompositions of $\Delta_{st}$, we get that
$\Delta_{st}$  is balanced.  Moreover the  only divisors  of $s^n$ are smaller
powers of $s$, of which only $s$ divides $\Delta_{s,t}$.

In cases (M2) and (M4) all divisors of $\Delta$ are balanced, in particular the
right-lcm  of two atoms. Moreover in any decomposition of $\Delta$ into a product of
atoms,  all atoms are different, hence the same property holds for divisors of
$\Delta$.  If $s$ is an atom we have thus $s^2\not\preccurlyeq\Delta$, so that
$\alpha(s^2)\neq  s^2$.  Since  $\alpha(s^2)$  has  length  at  most  2 and is
different  from $s^2$  it has  to be  of length  1, hence  equal to $s$ and by
induction  we  get  $\alpha(s^n)=\alpha(s\alpha(s^{n-1}))=\alpha(s^2)=s$. This
implies  that the left-gcd of $s^n$ with $\Delta_{s,t}$, which divides
$\alpha(s^n)$, is equal to $s$.

In the monoid $M$ for $G(e,e,r)$ introduced in \cite{CORRANPICANTIN}, any pair
of  dictinct atoms can be embedded in a submonoid $M'$ which is of type (M1) or
(M2)  (see \cite[section  6.3]{CALLEGAROMARIN}). This  embedding maps  right-lcms in
$M'$  to  right-lcms  in  $M$  and  left-divisibility  in $M'$ is the restriction of
left-divisibility in $M$ (see \cite[lemmas 5.1 and 5.3]{CALLEGAROMARIN}).

This implies that for properties (ii) and (iii) we are
either in the first or second situation, depending on the choice of the
couple of atoms .

The monoids (M5) have been investigated by M. Picantin in his thesis \cite{PICTHESE}.
Property (i) is clear, as the presentation is homogeneous. It is readily
checked that the left-gcd of two distinct atoms is the Garside element
$\Delta = x_1 x_2 \dots x_m x_1 \dots$ ($h$ terms), which proves (ii). Finally, (iii)
follows from the fact that every divisor of $\Delta$ but itself has
a unique decomposition as a product of atoms.

\end{proof}

Complex braid groups however may be the group of fractions of
Garside monoids which do not fulfill our conditions. An example,
pointed to us by M. Picantin, is given by the monoid presented by generators
$\sigma_1,\dots,\sigma_{n-1}$,  $\tau_1,\dots,\tau_{n}$ and
relations
 $
\sigma_i \sigma_{i+1} \sigma_i = \sigma_{i+1} \sigma_{i} \sigma_{i+1}$,
$\sigma_i \sigma_j = \sigma_j \sigma_i \mbox{\ for \ } |j-i| \geq 2$,
$\sigma_i \tau_i \tau_{i+1} = \tau_i \tau_{i+1} \sigma_i$,
$\sigma_i \tau_i = \tau_{i+1}\sigma_i$, $\sigma_i \tau_j = \tau_j \sigma_i$
for $j \not\in \{i,i+1 \}$. It is a Garside monoid for the Artin group
of type $B_n$, for which the right lcm $\sigma_i \tau_{i} = \tau_{i+1} \sigma_i$ of $\sigma_i$ and $\tau_{i+1}$
is clearly not balanced, hence condition (ii) is not satisfied. This presentation dates back to \cite{CHOW},
and the properties of the associated monoid are studied in \cite{PICCHOW}.

As pointed out to us by E. Godelle, there exist Garside monoids
satisfying (i) and (ii) but not (iii), such as
$M = <a, b \ | \ a^2 = b^2 >$, which provides a counterexample to the
next proposition. In \cite[proposition 4.4.3]{GODELLE} the next
proposition is proved for a classical Artin-Tits monoid (see also 
\cite[lemma 2.2]{GODELLE2}).

\begin{proposition}\label{rb=bt}  Let $M$ be a  Garside monoid satisfying (i),
(ii)  and (iii) above. Let $r$ be an atom  of $M$ and $b,z \in M$ be such that
$r^j b = b z$ for some $j \geq 1$; then $z = t^j$ for some atom $t$ with $rb =
bt$.
\end{proposition}
\begin{proofprop} 
Note  that  by properties  (i) and  (iii), given  two atoms $r\ne t$, the
element $\delta_{r,t}$ defined by $\Delta_{r,t}=r\delta_{r,t}$ satisfies
$\delta_{r,t}\ne 1$ and $r\not\preccurlyeq\delta_{r,t}$.

We first prove the following lemma.
\begin{lemma}  \label{lemgar1} Let $r$ be an atom in $M$ and $b \in M$ such
that  $r \not\preccurlyeq b$ and let $j\ge 1$ be such
that $r^j  b$ is divisible by some atom different from $r$;  
then there  exists an  atom $s\neq r$ such that $\delta_{r,s}
\preccurlyeq b$. 
\end{lemma}
\begin{proof} 

The proof is by induction on $j \geq 1$. Let $u \neq r$ be an atom
with $u \preccurlyeq r^j b$. Then $\Delta_{r,u}
\preccurlyeq r^j b$; writing $\Delta_{r,u} = r \delta_{r,u}$,
we get $\delta_{r,u} \preccurlyeq r^{j-1} b$. If $j = 1$ the conclusion holds
with $s=u$, and this covers the case $j = 1$.

In the case $j > 1$, since $\delta_{r,u} \neq 1$ and $r\not\preccurlyeq
\delta_{r,u}$, we have $v \preccurlyeq \delta_{r,u}$
for some atom $v \neq r$ hence $v \preccurlyeq r^{j-1} b$
and we can apply the induction assumption to $j-1< j$ and get the result.
\end{proof}

We also need the following lemma.

\begin{lemma} \label{lemgar2} Let $b \in M$ be a balanced element. Then, for any
atom $r \preccurlyeq b$ there exists an atom $t \preccurlyeq b$
with $rb = bt$, and $r \mapsto t$ provides a permutation of
the atoms dividing $b$.
\end{lemma}
\begin{proof}

For  $r \preccurlyeq b$ we write $b =  rb'$. Since $b$ is balanced we have $b'
\preccurlyeq  b$ hence $b = b' t$ for  some $t \in M$ (which obviously divides
$b$).  Thus $rb =  rb' t =  bt$ and $r  \mapsto t$ is well-defined and clearly
injective  from the set  of divisors of  $b$ to itself.  Because of the length
function  it maps atoms to atoms, and  it is surjective by the obvious reverse
construction.
\end{proof}

We show now the proposition by induction on the length of $b$, 
the case of length
$0$ being trivial. If $r \preccurlyeq b$ then, by simplifying $b$ by $r$,
we  get  the  result  by  induction.  We  thus  can  assume  $b \neq 1$ and $r
\not\preccurlyeq  b$. By lemma \ref{lemgar1} there is  an atom $s \neq r$ such
that  $\delta_{r,s} \preccurlyeq b$. Writing $b = \delta_{r,s} b'$ we get $r^j
\delta_{r,s}  b'  =  \delta_{r,s}  b'  z$.  As  $\Delta_{r,s}$ is balanced, it
conjugates  $r$ to some atom $u$,  by lemma \ref{lemgar2}, hence, cancelling a
power of $r$, we get $r \delta_{r,s} = \delta_{r,s} u$, whence
$\delta_{r,s}  u^j  b'  =  \delta_{r,s}  b'  z$. Cancelling $\delta_{r,s}$ and
applying  the induction assumption we get $z =  t^j$ for some atom $t$ with $u
b'  = b't$, so $rb = r \delta_{r,s} b' = \delta_{r,s} u b' = \delta_{r,s} b' t
= bt$. \end{proofprop}

\medskip    

We  first prove a general  corollary for a Garside group,  that is for the
group  of fractions $G$ of  a Garside monoid $M$.  It is a property of Garside
monoids that $G$ is generated as a monoid by $M$ and the powers of $\Delta$.

\begin{corollary}\label{center included}
Let $G$ be the group of fractions of a Garside monoid
$M$ satisfying properties (i), (ii) and (iii)
then if $U$ is a subgroup of $G$ containing a nontrivial power of each atom, 
the center of $U$ is contained in the center of $G$.
\end{corollary}
\begin{proof}
If  $x \in Z(U)$, then  for any atom $s$,  there exists a positive integer $N$
such  that  $x$  commutes  with  $s^N$.  Let  $i$  be  large  enough  so  that
$x\Delta^i\in  M$. Since conjugation  by $\Delta$ permutes  the atoms by lemma
\ref{lemgar2},  we have $s^Nx\Delta^i=xs^N\Delta^i=x\Delta^iu^N$ for some atom
$u$  such that $s\Delta^i =\Delta^iu$. Proposition \ref{rb=bt} applied with $b
=x  \Delta^i$, $r = s$, $j = N$  implies that $s x\Delta^i = x\Delta^i t$, for
some  atom $t$, hence we have $t^N=u^N$,  which gives $t=u$ by property (iii),
since  the left-gcd of $t^N=u^N$ with $\Delta_{t,u}$ has  to be equal to $t$ and to
$u$. Hence $s x\Delta^i = x\Delta^i u$ and $sx = xs$.
\end{proof}

A particular case of the above corollary is the following, using the fact that
for  monoids of type (M4) in proposition \ref{monoid list}, the squares of the
atoms lie in the pure group.

\begin{corollary}
The  center of  pure Artin-Tits  groups of  type $\tilde  A$ or  $\tilde C$ is
trivial.
\end{corollary}

Another application of proposition \ref{rb=bt} gives theorem \ref{thff}
for all finite complex reflection groups except $G_{31}$. 

\begin{corollary}  Assume  $W$  is  a  complex reflection group different from
$G_{31}$  and let $U$  be a finite  index subgroup of  $B$. Then $Z(U) \subset
Z(B)$.
\end{corollary}
\begin{proof}
It is well-known
that all possible braid groups $B$ can be obtained from a 2-reflection
group, so we can restrict to this case.
Corollary \ref{center included} gives the result for all complex braid group
for which a monoid satisfying (i), (ii) and (iii) is known. 
This covers all of them except for a complex reflection group of type $G_{31}$ and the infinite series
$G(de,e,r)$ for $d>1$ and $e>1$. Indeed, one can use
\begin{itemize}
\item the classical monoid for Coxeter groups and Shephard groups
(groups which have same $X/W$ as
a Coxeter group), as well as for $G_{13}$, which has the same braid group
as the Coxeter group $I_2(6)$ 
\item the parachute monoid for the $G(e,e,r)$
\item the dual monoid for the $G(e,e,r)$, $G(d,1,r)$ and the exceptional
groups of rank at least $3$ which are not $G_{31}$
\item the monoids $f(4,3)$ and $f(5,3)$ for $G_{12}$ and $G_{22}$.
\end{itemize}
Since the only 2-reflection exceptional groups
are $G_{12}$, $G_{13}$ and $G_{22}$ this indeed covers everything
but $G_{31}$ and the $G(de,e,r)$, for $d > 1, e>1$.

But the braid group associated with $G(de,e,r)$ is a subgroup of
finite  index in the  braid group associated  to $G(de,1,r)$ (and  thus
of the
classical braid group on $r$ strands), whence the result in this case.
\end{proof}

\section{Springer theory and $G_{31}$}

We let $B_{n}$ denote the braid group associated to $G_{n}$,
and $P_n = \mathrm{Ker} (B_n \to G_n)$ the corresponding pure braid group.
In particular $B_{37}$ denotes the Artin-Tits group of type $E_8$.
By Springer theory (see \cite{SPRINGER}), 
$G_{31}$ appears as the centralizer of
a regular element $c$ of order $4$ in $G_{37}$, and, 
as a consequence of \cite[thm. 12.5 (iii)]{BESSIS},
$B_{31}$ can be identified with the centralizer of a lift
$\tilde{c} \in B_{37}$ of $c$, in such a way that the natural diagram
$$
\xymatrix{
B_{31}\ar[r]\ar[d] &   B_{37}\ar[d] \\
G_{31}\ar[r] &   G_{37} \\
}$$
commutes.  The proof of theorem \ref{thzb} for $G_{31}$ was communicated to us
by  D. Bessis. For  the convenience of  the reader we  reproduce it here.
We let $\pi$ denote the positive generator of $Z(P_{37})$.
We have $\tilde{c}^4 = \pi$. On the other hand, it can be checked that $G_{37}$
has a regular element of order 24, and as a consequence of \cite{BROUE-MICHEL}
there exists $\tilde{d} \in B_{37}$ with $\tilde{d}^{24} = \pi$ whose image in
$G_{37}$ is a regular element $d$ of order $24$.
Moreover, by \cite[12.5 (ii)]{BESSIS},
$\tilde{d}^6$  is conjugated to $\tilde{c}$,  so up to conjugating $\tilde{d}$
we  can assume  $\tilde{d}^6 =  \tilde{c}$, and  in particular  $\tilde{d} \in
B_{31}$.  So $Z(B_{31})$ lies inside the  centralizer of $\tilde{d}$, which is
by  another application  of \cite[12.5  (iii)]{BESSIS} the  braid group of the
centralizer  in $G_{37}$  of $d$.  Since $24$  divides exactly  one reflection
degree of $G_{37}$, Springer's theory says that $d$ generates its centralizer,
as  its centralizer  is a  reflection group  whose single reflection degree is
$24$.  This implies  that the  centralizer of  $\tilde d$  is the cyclic group
generated  by $\tilde d$, thus any $x  \in Z(B_{31})$ is a power $\tilde{d}^a$
of  $\tilde{d}$. Since its image  in $G_{37}$ should lie  in $Z(G_{31}) = <d^6
>$,  $Z(B_{31}) = < \tilde{d}^6 >$ is a cyclic group, isomorphic to $\Z$ as it
is  infinite,  for  instance  because  $\beta_{31} \in Z(B_{31})$ has infinite
order.   But  $\beta_{31}   \in  Z(B_{31})   =  <   \tilde{d}^6  >$  satisfies
$\beta_{31}^4  = \pi_{31} = \pi_{37}  = (\tilde{d}^6)^4$, hence $\tilde{d}^6 =
\beta_{31}$  and $Z(B_{31})  = <  \beta_{31} >  \simeq \Z$. This concludes the
proof of theorem \ref{thzb}.

 





In order to prove theorem \ref{thff}, we will need an explicit
description of this embedding $B_{31} \hookrightarrow B_{37}$.
A presentation for $B_{37}$ and $B_{31}$ is given by the
Coxeter-like diagrams
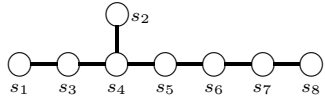
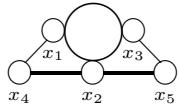
\def\nnode#1{{\kern -0.6pt\mathop\bigcirc\limits_{#1}\kern -1pt}}
\def\bar{{\vrule width10pt height3pt depth-2pt}}
\def\lbar{{\vrule width19pt height3pt depth-2pt}}
\def\vertbar#1#2{\rlap{\kern4pt\vrule width1pt height17pt depth-7pt}
 \rlap{\raise19pt\hbox{$\kern -0.4pt\bigcirc\scriptstyle#2$}}
                 \nnode{#1}}
\font\mediumcirc=lcircle10 scaled \magstep 2
\def\sBcirc{$\hbox{\mediumcirc\char"6E}$}
$$\nnode{s_1}\bar\nnode{s_3}\bar\vertbar
{s_4}{s_2}\bar\nnode{s_5}\bar\nnode{s_6}\bar\nnode{s_7}\bar\nnode{s_8}$$
and
$$\nnode{x_4}\kern -2pt\raise8pt\hbox{$\diagup$}
      \kern -2.5pt\raise16pt\hbox{$\nnode{x_1}$}
      \kern-12.5pt\lbar\nnode{x_2}
   \kern-3.7pt\raise17.8pt\hbox{$\sBcirc$}
   \kern-17.8pt\lbar\nnode{x_5}
   \kern-20.5pt\raise16pt\hbox{$\nnode{x_3}$}
   \kern-2pt\raise8pt\hbox{$\diagdown$}$$
where the circle means $x_1 x_2 x_3 = x_2 x_3 x_1 = x_3 x_1x_2$. This latter
presentation was conjectured in \cite{BMR} and proved in \cite{BESSIS}.

We choose for regular element $\tilde{c} =  (s_4s_2s_3s_1s_4s_3s_5s_6s_7s_8)^6$,
and we use the algorithms of \cite{FRANCOMENESES} included in the 
{\tt GAP3} package {\tt CHEVIE} (see \cite{CHEVIE}) to find generators
for $B_{31}$, considered as the centralizer of $\tilde{c}$ in $B_{37}$.
From this we get a description of the embedding $B_{31} \hookrightarrow B_{37}$ as follows:

$$
\begin{aligned}
x_1  &\mapsto (s_2 s_3 s_1 s_5)^{-1}  s_1 s_4 (s_2 s_3 s_1  s_5) \\
x_2 &\mapsto (s_4 s_2 s_3  s_5 s_6 s_5 s_7) ^{-1}  s_2 s_5 (s_4 s_2 s_3  s_5 s_6 s_5 s_7) \\
x_3 &\mapsto (s_5  s_6 s_7)^{-1} s_1 s_4 (s_5 s_6 s_7) \\
x_4 &\mapsto (s_2 s_5 s_6)^{-1} s_4 s_6 (s_2  s_5 s_6) \\
x_5 &\mapsto  (s_3 s_1 s_5 s_6)^{-1}s_4 s_8  (s_3 s_1 s_5 s_6) \\
\end{aligned}
$$
One can check that $\beta_{31} =   (x_4 x_1
x_2 x_3 x_5)^6 \mapsto \tilde{c}$.


\section{Krammer representations and $G_{31}$}

We   use   the   generalized   Krammer   representation   $\hat{R}   :  B_{37}
\hookrightarrow  \GL_{120}(\Q[q,q^{-1},t,t^{-1}])$ for the Artin-Tits group of type
$E_8$ defined in \cite{DIGNE} and \cite{COHENWALES}. We use the definitions of
\cite{DIGNE}  (note  however  the  erratum  given  in \cite{KRAMINF}). One has
$\beta_{37} = \beta_{31}^2$ and $\hat{R}(\beta_{37}) = q^{30} t$. The standard
generators   $s_i$  of  $B_{37}$  are   mapped  to  semisimple  matrices  with
eigenvalues  $q^2t$  (once),  $-q$  (28  times)  and  $1$ (91 times). We embed
$\Q[q^{\pm  1},t^{\pm  1}]$  into  $\Q[q^{\pm  1},u^{\pm 1}]$ under $t \mapsto
u^2$.  Then  $R(\beta_{31})$  can  be  diagonalized,  has  two  60-dimensional
eigenspaces, corresponding to the eigenvalues $\pm q^{15} u$. An explicit base
change,  that  we  choose  in  $\GL_{120}(\Q[q^{\pm  1},u^{\pm 1})$ so that it
specializes  to the identity when $q \mapsto 1, u \mapsto 1$, provides another
faithful representation $R : B_{37} \to \GL_{120}([q^{\pm 1},u^{\pm 1}])$ with
$R(B_{31}) \subset \GL_{60}(\Q[q^{\pm 1},u^{\pm 1}]) \times \GL_{60}(\Q[q^{\pm
1},u^{\pm  1}])$. Let $U$ be  a finite index subgroup  of $B_{31}$, and let $N
\geq  1$ such that $x_i^{2N} = (x_i^2)^N \in  U$ for all $i$. Note that, since
the  images of  the $x_i$'s  in $G_{31}$  have order  $2$, we  have $x_i^2 \in
P_{31}$.

We prove that the centralizer of $R(U)$ in $\GL_{60} \times \GL_{60}$ is equal
to  the centralizer of  $R(B_{31})$. This is  the case as  soon as $R_j(K U) =
\Mat_{60}(K)$, where $j \in \{1, 2 \}$, where $K$ is an arbitrary extension of
$\Q(q,u)$, $R = R_1 \times R_2$ is the obvious decomposition, and $K U$ is the
group  algebra of  $U$. Actually,  it is  clearly enough  to prove  this for a
\emph{specialization}  of $R_j$ to  given values of  $t,u$. More precisely, if
$A$  is a unital  ring with field  of fractions $K$  and if we have a morphism
$\Q[q^{\pm  1},u^{\pm  1}]  \to  A$,  letting  $R'_j : B_{31} \to \GL_{60}(A)$
denote  the induced  representation we  have that  $R'_j(K U)  = \Mat_{60}(K)$
implies that $R_j(\Q(q,u)U) = \Mat_{60}(\Q(q,u))$.

Let $j \in \{ 1, 2 \}$.
We  use  the  morphism  $\Q[q^{\pm  1},u^{\pm  1}]  \to \Q[[h]]$ which maps $q
\mapsto  e^h$, $u \mapsto e^{7h}$ (the choice of $7$ being rather random), and
denote  $R'_j : P_{31} \to  \Mat_{60}(\Q[[h]])$ the induced representation. We
prove that the unital algebra generated over $\Q((h))$ by the $R'_j(x_i^{2N})$
full  $\Mat_{60}(\Q((h)))$.  As  $x_i^2  \in  P_{31}  \subset  P_{37}$ we have
$R'_j(x_i^2)   \equiv   1$   modulo   $h$   (see  e.g.  \cite{KRAMINF})  hence
$R'_j(x_i^2)^N  \equiv 1$  modulo $h$.  We define  $y_i \in  \Mat_{60}(\Q)$ by
$h^{-1}(R'_j(x_i^2)-1)  \equiv y_i$ modulo  $h$ and thus $h^{-1}(R'_j(x_i^2)^N
-1)  \equiv N y_i$. By Nakayama's lemma it is now sufficient to prove that the
$y_i$  generate  $\Mat_{60}(\Q)$.  It  turns  out  that  the  $y_i$  belong to
$\Mat_{60}(\Z)$.  By another application of  Nakayama's lemma it is sufficient
to  check that the reduction mod  $p$ of the $y_i$ generates $\Mat_{60}(\F_p)$
for  some prime  $p$. For  a given  $p$, the  determination by computer of the
dimension  of  the  subalgebra  of  $\Mat_{60}(\F_p)$  generated  by  elements
$y_1,\dots,y_5$  is  easy  :  starting  from  the line $F_0 = \F_p \mathrm{Id}
\subset  \Mat_{60}(\F_p)$ we compute the vector space $F_{r+1} = F_r + F_r y_1
+  \dots F_r y_5$ until $\dim F_r = \dim F_{r+1}$. For $p = 37$ we get $3600 =
\dim \Mat_{60}(\F_p)$, thus concluding the argument.


By  faithfulness of  $R$ this  proves that  $Z(U)$ commutes with all $B_{31}$,
thus  $Z(U)  \subset  Z(B_{31})$,  and  this  concludes  the  proof of theorem
\ref{thff}.

\section{Miscellaneous consequences}

We  notice that the representations  $R_1$ and $R_2$ are  deduced one from the
other  under the field automorphism of $\Q(q,u)$  defined by $q \mapsto q$, $u
\mapsto   -u$.  Since   $R$  is   faithful,  this   implies  that 
the representations $R_j$ are faithful representations of $B_{31}$.
We  note that the $R_j$ have dimension  60, which is the number of reflections
in  $G_{31}$,  and  thus  the  dimension  of  the  representation  involved in
conjecture 1 of \cite{KRAMCRG}. 

This enables us to prove this conjecture
for $G_{31}$, that is the following theorem:

\begin{theorem} $B_{31}$ can be embedded into $\GL_{60}(K)$ as a Zariski-dense
subgroup, for $K$ a field of characteristic 0, and $P_{31}$ is residually
torsion-free nilpotent.
\end{theorem}

\begin{proof}
Since $P_{31}$ embeds in $P_{37}$, the statement about residual torsion-free
nilpotence is a consequence of the corresponding statement for
$P_{37}$, proved in \cite{KRAMINF}, \cite{RESNIL}.
Embedding  $\Q(q,u)$  into  $\Q((h))$  through  $q  \mapsto  e^h$,  $u \mapsto
e^{\alpha  h}$, for $\alpha$ a transcendent number, we can assume $R_j(B_{31})
\subset \GL_{60}(\Q[\alpha]((h)))$.

In  addition  we  know  that  $R_j(x_i^2)  =  \exp(  h a_i)$ for some $a_i \in
\gl_{60}(\Q[\alpha][[h]])$  with  $a_i  \equiv  \tilde{y}_i  \mod  h$ for some
$\tilde{y}_i \in \gl_{60}(\Q[\alpha])$ which specialize to $y_i$ under $\alpha
\mapsto  7$. A computer  calculation similar to  the above shows  that the Lie
algebra  generated  by  the  $y_i$  is  $\gl_{60}(\Q)$,  and thus that the Lie
algebra  generated  by  the  $a_i$  over  $\C((h))$ is $\gl_{60}(\C((h)))$. By
\cite{LIETRANSP}   lemme  21   this  proves   that  the   Zariski  closure  of
$R_j(B_{31})$  contains  $\GL_{60}(\C((h)))$,  which  proves
the theorem.
\end{proof}

We remark that this conjecture, if true  for all
braid  groups, would immediately  imply theorem \ref{thff}.
  Note however that
the   representations  $R_j$  have  same  dimension  as,  but  are  \emph{not}
isomorphic  to the representation  constructed in \cite{KRAMCRG}.  This can be
seen  from the eigenvalues of the generators : since $R(x_1)$ is conjugated to
$R(s_1s_4)$  and since $s_1 s_4 = s_4  s_1$, the 3 eigenvalues of the $R(s_i)$
provide  at most 9 eigenvalues  for $R(x_1)$, and by  checking each of them we
find  that
$R(x_1)$ has  for eigenvalues $q^2t$ (twice), $q^2$ (6  times),
$-q$ (44 times) and $1$
(68  times). As a consequence,  the $R_j(x_1)$ both have  4 eigenvalues, to be
compared  with the  3 eigenvalues  the generators  have in the construction of
\cite{KRAMCRG}.




\end{document}